\documentclass[12pt,a4paper]{amsart}
\usepackage{latexsym}
\usepackage{amsfonts}
\usepackage[mathscr]{eucal}
\usepackage{epsfig}
\usepackage{a4wide}
\usepackage{ifthen}

\newboolean{draft}
\setboolean{draft}{false}

\def\@copyright{}
\newtheorem{theorem}{Theorem}[section]

\newtheorem{example}{Example}[section]

\newtheorem{corollary}{Corollary}[section]
\newtheorem{remark}{Remark}[section]

\newcommand{\N}{ \mathbb{N} }
\newcommand{\C}{ \mathbb{C} }

\newcommand{\R}{ \mathbb{R} }

\newcommand{\trunc}[1]{ {\lfloor #1 \rfloor} }

\newcommand{\calB}{\mathcal{B}}

\newcommand{\calD}{\mathcal{D}}
\newcommand{\calE}{\mathcal{E}}
\newcommand{\calF}{\mathcal{F}}

\newcommand{\calN}{\mathcal{N}}

\newcommand{\calZ}{\mathcal{Z}}

\newcommand{\vecnull}{{\mathbf 0}}

\newcommand{\Var}{{\mbox{Var\,}}}
\newcommand{\Cov}{{\mbox{Cov\,}}}

\begin{document}

\pagestyle{plain}
\begin{titlepage}
\pagenumbering{arabic}
\title[]{}
\date{11/2006}
\end{titlepage}

\maketitle

\begin{center}
  \Large
  WEIGHTED DICKEY-FULLER PROCESSES FOR DETECTING STATIONARITY
\end{center}
\vskip 1cm
\begin{center}
  Ansgar Steland\footnote{Address of correspondence: Prof. Dr. Ansgar Steland,
Institute of Statistics, RWTH Aachen University, W\"ullnerstr.~3, D-52056 Aachen, Germany.}\\
  Institute of Statistics, \\
  RWTH Aachen University, Germany\\
  steland@stochastik.rwth-aachen.de
\end{center}

\begin{abstract}
Aiming at monitoring a time series to detect stationarity as soon
as possible, we introduce monitoring procedures based on
kernel-weighted sequential Dickey-Fuller (DF) processes, and
related stopping times, which may be called weighted Dickey-Fuller
control charts. Under rather weak assumptions, (functional)
central limit theorems are established under the unit root null
hypothesis and local-to-unity alternatives. For general dependent
and heterogeneous innovation sequences the limit processes depend
on a nuisance parameter. In this case of practical interest, one
can use estimated control limits obtained from the estimated
asymptotic law. Another easy-to-use approach is to transform the
DF processes to obtain limit laws which are invariant with respect
to the nuisance parameter. We provide asymptotic theory for both
approaches and compare their statistical behavior in finite
samples by simulation.
\\[1cm]
\end{abstract}

\newpage
\section*{Introduction}

Analyzing whether a time series is stationary or is a non-stationary
random walk (unit root process) in the sense that the first order
differences form a stationary series is an important issue in time
series analysis, particularly in econometrics. Often the task is to
test the unit root null hypothesis against the alternative of
stationarity at a pre-specified $ \alpha $ level, which ensures that
a decision in favor of stationarity is statistically significant.
For instance, the equilibrium analysis of macroeconomic variables as
established by Granger (1981) and Engle and Granger (1987) defines
an equilibrium of two random walks as the existence of stationary
linear combination. When analyzing equilibrium errors of a
cointegration relationship, rejection of the null hypothesis in
favor of stationarity means that the decision to believe in a valid
equilibrium is statistically justified at the pre-specified $\alpha
$ level. For an approach where CUSUM based residual tests are
employed to test the null hypothesis of cointegration, we refer to
Xiao and Phillips (2002). Their test uses residuals calculated from
the full sample. In the present article we study sequential
monitoring procedures which aim at monitoring a time series until a
time horizon $T$ to detect stationarity as soon as possible.

The question whether a time series is stationary or a random walk is
also of considerable importance to choose a valid method when
analyzing the series to detect trends. Such procedures usually
assume stationarity, see Steland (2004, 2005a), Pawlak et al.
(2004), Huskov\'a (1999), Huskov\'a and Slab\'y (2001), Ferger
(1993, 1995), among others. As shown in Steland (2005b), when using
Nadaraya-Watson type smoothers to detect drifts the limiting
distributions for the random walk case differ substantially from the
case of a stationary time series.

To detect changes in a process or a misspecified model, a common approach
originating in statistical quality control is to formulate an in-control model
(null hypothesis) and an out-of-control model (alternative), and to apply
appropriate control charts resp. stopping times. Given a time series $ Y_1, Y_2, \dots $
a monitoring procedure with time horizon (maximum sample size) $T$
is given by a stopping time
$
  S_T^* = \inf \{ 1 \le t \le T : U_t \in A \}
$ using the convention $ \inf \emptyset = \infty $, where $ U_t $,
called {\em control statistic}, is a $ \sigma(Y_1, \dots, Y_t)
$-measurable $ \R $-valued statistic sensitive for the alternatives
of interest, and $ A \subset \R $ is a measurable set such that $ \{
U_t \in A \} $ has small probability under the null model and high
probability under the alternative of interest. In most cases $ A $
is of the form $ (-\infty,c) $ or $ (c,\infty) $ for some given {\em
control limit} (critical value) $ c $. To design monitoring
procedures, the standard approach is to choose the control limit to
ensure that the average run length (ARL), $ ARL = E( S_T^* ) $, is
greater or equal to some pre-specified value. However, controlling
the significance level is a also serious concern. The results
presented in this article can be used to control any characteristic
of interest, although we will focus on the type I error in the
sequel.

The (weighted) Dickey-Fuller control chart studied in this article
is essentially based on a sequential version of the well-known
Dickey-Fuller (DF) unit root test, which is motivated by least
squares. Due to its power properties this test is very popular,
although it is known that its statistical properties strongly depend
on a correct specification of the correlation structure of the
innovation sequence. The DF test and  its asymptotic properties,
particularly its non-standard limit distribution have been studied
by White (1958), Fuller (1976), Rao (1978, 1980), Dickey and Fuller
(1979), and Evans and Savin (1981), Chan and Wei (1987, 1988),
Phillips (1987), among others. We will generalize some of these
results. To ensure quicker detection in case of a change to
stationarity, we modify the DF statistic by introducing kernel
weights to attach small weights to summands corresponding to past
observations. We provide the asymptotic theory for the related
Dickey-Fuller (DF type) processes and stopping times, also covering
local-to-unity alternatives.

For correlated error terms the asymptotic distribution of the DF
test statistic, and hence the control limit of a monitoring
procedure, depends on a nuisance parameter, which can be estimated
by Newey-West type estimators. We consider two approaches to deal
with that problem. Firstly, based on a consistent estimate of the
nuisance parameter one may take the asymptotic control limit
corresponding to the estimated value. Secondly, following Phillips
(1987) one may consider appropriate transformations of the processes
possessing limit distributions which no longer dependent on the
nuisance parameter. A nonparametric approach called KPSS test which
avoids this problem, at least for I(1) processes, has been proposed
by  Kwiatkowski et al. (1992). That unit root test has better type I
error accuracy, but tends to be less powerful. Monitoring procedures
related to this approach and their merits have been studied in
detail in Steland (2006).

The organisation of the paper is as follows. In
Section~\ref{SecModel} we explain and motivate carefully our
assumptions on the time series model, and present the class of
Dickey-Fuller type processes and related stopping times. The
asymptotic distribution theory under the null hypothesis of a random
walk is provided in Section~\ref{AsH0}. Section~\ref{AsH1} studies
local-to-unity asymptotics, where the asymptotic distribution is
driven by an Ornstein-Uhlenbeck process instead of the Brownian
motion appearing in the unit root case. Finally, in
Section~\ref{SecSim} we compare the methods by simulations.

\section{Model, assumptions, and Dickey-Fuller type processes and control charts}
\label{SecModel}

\subsection{Time series model}

Our results work under quite general nonparametric assumptions allowing for dependencies and
conditional heteroskedasticity (GARCH effects), thus providing a nonparametric view
on the parametrically motivated approach. To motivate our assumptions, let
us consider the following common time series model, which is often used in applications.
Suppose at this end that $ \{ Y_t \} $ is an AR($p$) time series, i.e.,
$$
  Y_t = \alpha_1 Y_{t-1} + \dots + \alpha_p Y_{t-p} + u_t,
$$
for starting values $ Y_{-p}, \dots, Y_{-1} $,
where $ \{ u_t \} $ are i.i.d. error terms (innovations) with $ E(u_t) = 0 $ and
$ \sigma_u^2 = \Var( u_t ) $, $ 0 < \sigma_u^2 < \infty $.
Assume the characteristic polynomial
$$
  p(z) = 1 - \alpha_1 z - \dots - \alpha_p z^p, \qquad z \in \C,
$$
has a unit root, i.e., $ p(1) = 0 $, of multiplicity $1$, and all other
roots are outside the unit circle, i.e., $ p(z) = 0 $ implies $ |z| > 1 $.
Then $ p(z) = p^*(z)(1-z) $ for some polynomial $ p^*(z) $ with has no roots
in the unit circle implying that $ 1/p^*(z) $ exists for all $ |z| \le 1 $.
We obtain $ p(L) = p^*(L) \Delta Y_t = \epsilon_t $,
where $L$ denotes the lag operator. Since $p^*(L) $ can be inverted, we have
the representation
\begin{equation}
\label{ARpAsAr1}
  Y_t = Y_{t-1} + \sum_{j \ge 0} \beta_j u_{t-j},
\end{equation}
for coefficients $ \{ \beta_j \}. $ This means, $ \{ Y_t \} $
satisfies an AR($1$) model with correlated errors. For the
calculation of $ \beta_j $ we refer to Brockwell and Davis (1991,
Sec. 3.3.) In particular, to analyze an AR($p$) series for a unit
root, one can work with an AR($1$) model with correlated errors.

The representation (\ref{ARpAsAr1}) motivates the following time series framework
which will be assumed in the sequel. Suppose we are given an univariate time series
$ \{ Y_t : t = 0,1, \dots \} $ satisfying
\begin{equation}
\label{ARModel}
  Y_t = \rho Y_{t-1} + \epsilon_t, \qquad t \ge 1, \qquad Y_0 = 0,
\end{equation}
where $ \rho \in (-1,1] $ is a fixed but unknown parameter.
Concerning the error terms $ \{ \epsilon_t \} $ we impose the following
assumptions.
\begin{itemize}
  \item[(E1)] $ \{ \epsilon_t \} $ is a strictly stationary series with mean
  zero and $ E |\epsilon_1|^4 < \infty $ with the following properties:
  We have
  $$
    \sum_{t=1}^\infty \Cov( \epsilon_1^2, \epsilon_{1+t}^2 ) <
    \infty,
  $$
  and both $ \{ \epsilon_t \} $ and
  $ \{ \epsilon_t^2 \} $ satisfy a functional central limit theorem,
  i.e.,
  \begin{equation}
  \label{FCLT1}
    T^{-1/2} \sum_{i \le \trunc{Ts}} \epsilon_i \Rightarrow \eta B(s),
  \end{equation}
  and
  \begin{equation}
  \label{FCLT2}
    T^{-1/2} \sum_{i \le \trunc{Ts}} (\epsilon_i^2 - E \epsilon_i^2 )  \Rightarrow \eta' B^{(2)}(s),
  \end{equation}
  as $ T \to \infty $, for constants $ 0 < \eta, \eta' < \infty  $.
  Here $ B $ and $ B^{(2)} $ denote (standard) Brownian motions (Wiener processes) with start in $0$.
  \item[(E2)] $ \{ \epsilon_t \} $ is a strong mixing strictly stationary times series
  with $ E|\epsilon_1|^{4(1+\delta)} < \infty $ for some $ \delta > 0 $, and
  with mixing coefficients, $ \alpha(k) $, satisfying
  $$
    \sum_{k=1}^\infty (k+1)^\delta \alpha(k)^{\delta/(2+3\delta)} < \infty.
  $$
\end{itemize}
In assumption (E1) and the rest of the paper
$ \Rightarrow $ denotes weak convergence in the
space $ D[0,1] $ of all cadlag functions equipped with the Skorokhod metric $ d $.

\begin{remark}
The assumption that $ \{ \epsilon_t \} $ satisfies an invariance
principle can be regarded as a nonparametric definition of the $
I(0) $ property ensuring that the partial sums converge weakly to a
(scaled) Brownian motion $ B $. For a parametrically oriented
definition see Stock (1994). Particularly, the scale parameter $
\eta $ is given by
\begin{equation}
\label{DefEta2}
  \eta^2 = \lim_{T \to \infty} \eta_T^2, \qquad
  \eta_T^2
           = \sigma^2 + 2 \sum_{t=1}^T (T-t)T^{-1} E( \epsilon_1 \epsilon_{1+t}
           )
\end{equation}
Also introduce the notations
\begin{equation}
\label{DefF}
  \vartheta_T^2 = \eta_T^2 / \sigma^2, \qquad
  \vartheta = \lim_{T \to \infty} \vartheta_T.
\end{equation}
If the $ \epsilon_t $ are uncorrelated, we have $ \eta_T^2 = \sigma^2 $, and
 $ \vartheta_T^2 = 1 $.
\end{remark}

As a non-trivial example for processes satisfying (E1) let us
consider ARCH processes.

\begin{example} A time series $ \{ X_t \} $ satisfies ARCH($\infty$) equations,
if there exists a sequence of i.i.d. non-negative random variables, $ \{ \xi_t \} $,
such that
$$
  X_t = \rho_t \xi_t, \qquad \rho_t = a + \sum_{j=1}^\infty b_j X_{t-j}
$$
where $ a \ge 0 $, $ b_j \ge 0 $, $ j = 1,2, \dots $ This model is often applied
to model conditional heteroscedasticity of an uncorrelated sequence
$ \{ \epsilon_t \} $ with $ E\epsilon_t = 0 $ for all $t$, by putting
$ X_t = \epsilon_t^2 $. A common choice for $ \xi_t $ is to assume that the
$ \xi_t $ are i.i.d. with common standard normal distribution. In Giraitis et al. (2003)
it has been shown that an unique and strictly stationary solution exists and satisfies
$ \sum_k \Cov( X_1, X_{1+k} ) < \infty $, if
$$
  ( E \xi_1^2 )^{1/2} \sum_{j=1}^\infty b_j < 1.
$$
In addition, under these conditions the functional central limit
theorem (\ref{FCLT2}) holds. The rate of decay of the coefficients $
b_j $ controls the asymptotic behavior of $ \Cov( X_1, X_{1+k} ) $.
If for some $ \gamma > 1 $ and $ c > 0 $ we have $ b_j \le c
j^{-\gamma} $, $ j = 1,2,\dots $, then there exists $ C > 0 $ such
that $ \Cov( X_1, X_{1+k} ) \le C k^{-\gamma} $ for $ k \ge 1 $.
Thus, depending on the rate of decay (E2) may also holds.
\end{example}

\begin{remark} Assumption (E2) will be used to verify a tightness
criterion. Combined with appropriate moment conditions it implies
the invariance principles (\ref{FCLT1}) and (\ref{FCLT2}).
\end{remark}

\subsection{Dickey-Fuller processes}

We will now introduce the class of Dickey-Fuller processes and related
detection procedures.
Recall that the least squares estimator of the parameter $ \rho $ in
model (\ref{ARModel}) is given by
$$
  \widehat{\rho}_T = \sum_{t=1}^T Y_{t-1} Y_t \ \Big/ \ \sum_{t=1}^T Y_{t-1}^2.
$$
To test the null hypothesis $ H_0: \rho = 1 $, one forms the
Dickey-Fuller (DF) test statistic
$$
  D_T = T( \widehat{\rho}_T - 1 ) = \frac{T^{-1} \sum_{t=1}^T Y_{t-1} (Y_t-Y_{t-1}) }
  { T^{-2} \sum_{t=1}^T Y_{t-1}^2 },
$$
Suppose at this point that the $ \epsilon_t $ are uncorrelated.
Provided $ | \rho | < 1 $,
$
  \left\{ \sum_{t=1}^T Y_{t-1}^2 \right\}^{1/2} (\widehat{\rho}_T - 1)
    \stackrel{d}{\to} \calN(0,1),
$
as $ T \to \infty $. However, $ \widehat{\rho}_T $ has
a different convergence rate and a non-normal limit distribution, if $ \rho = 1 $.
It is known that
$$
  D_T \stackrel{d}{\to} \calD_1 = (1/2) ( B(1)^2 - 1 ) \ \Big/ \ \int_0^1 B(r)^2 dr,
$$
as $ T \to \infty $, see White (1958), Fuller (1976), Rao (1978, 1980),
Dickey and Fuller (1979), and Evans and Savin (1981).
Recall that $ B $ denotes standard Brownian motion.
Based on that result one can
construct a statistical level $ \alpha $ test, which rejects the null hypothesis
$ H_0: \rho = 1 $ of a unit root against the alternative $ H_1 : \rho < 1 $
if $ D_T < c $, where the critical value $ c $ is the $ \alpha $-quantile of
the distribution of $ \calD_1 $.
More generally, we want to construct a detection rule which provides a signal
if there is some change-point $q$ such that $ Y_1, \dots, Y_{q-1} $ form a
random walk (unit root process), and $ Y_q, \dots, Y_T $ form an $ AR(1) $
with dependent innovations. This means, the alternative hypothesis is
$ H_1 = \cup_{1 \le q \le T} H_1^{(q)} $, where $ H_1^{(q)} $, $ 1 \le q \le T $, specifies that
$$
  Y_{t} = \left\{
    \begin{array}{ll}
      Y_{t-1} + \epsilon_t,      & 1 \le t < q, \\
      \rho Y_{t-1} + \epsilon_t, & q \le t \le T,
    \end{array}
    \right.
$$
where $ \rho \in (-1,1) $. However, for the calculation of the detection rule to be introduced
now knowledge of a specific alternative hypothesis is not required.

A naive approach to monitor a time series to check for
deviations from the unit root hypothesis
is to apply the DF statistic at each time point using the most recent
observations.
A more sophisticated version of this idea is to modify the DF statistic
to ensure that summands in the numerator have small weight if their time
distance to the current time point is large. To define such a detection rule,
let us introduce the following sequential kernel-weighted
Dickey-Fuller (DF) process
\begin{equation}
\label{DefDT}
  D_T(s) =
  \frac{ \trunc{Ts}^{-1} \sum_{t=1}^{ \trunc{Ts} } Y_{t-1} \Delta Y_t  K( ( \trunc{Ts} - t )/ h ) }
       { \trunc{Ts}^{-2} \sum_{t=1}^{ \trunc{Ts} } Y_{t-1}^2 }, \qquad s \in [0,1],
\end{equation}
where $ \Delta Y_t = Y_t - Y_{t-1} $. Here and in the following we
put $ 0/0 = 0 $ for convenience. Note that $ \trunc{Ts} $ plays
the role of the current time point. The non-negative smoothing
kernel $K$ is used to attach smaller weights to summands from the
distant past to avoid that such summands dominate the sum. Thus,
kernels ensuring that $ z \mapsto K(|z|) $, $ z \in \R $, is
decreasing are appropriate, but that property is not required. We
do not use kernel weights in the denominator, since it is used to
estimate a nuisance parameter. We will require the following
regularity conditions for $ K : \R \to \R_0^+ $.
\begin{itemize}
  \item[(K1)] $ \| K \|_\infty < \infty $, $\int K(z) dz = 1 $ and $ \int z K(z) dz = 0 $.
  \item[(K2)] $ K $ is $ C^2 $ with bounded derivative.
  \item[(K3)] $ K $ has bounded variation.
\end{itemize}
Note that it is not required to use a kernel with compact support.

The parameter $ h = h_T$ is used as a scaling constant in the kernel and defines
the memory of the procedure. For instance,
if $ K(z) > 0 $ if $ z \in [-1,1] $ and $ K(z) = 0 $ otherwise, the process
$ U_T $ looks back $ h $ observations. We will assume that
\begin{equation}
\label{AssZeta}
  T / h_T \to \zeta, \qquad T \to \infty,
\end{equation}
for some $ 1 \le \zeta < \infty $.
That condition ensures that the number of observations used by $ D_T $
gets larger as $ T $ increases. Note that the parameter $ \zeta $, which will also
appear in the limit distributions, could be absorbed into the kernel $K$.
However, in practice one usually fixes a kernel $K$ and chooses a bandwidth $ h $
relative to the time horizon $T$. (\ref{AssZeta}) is therefore not restrictive.

\subsection{Dickey-Fuller type control charts}

Since small values of $ D_T(s) $ provide evidence for the alternative
that the time series is stationary, intuition suggests that the control chart
should give a signal if $ D_T $ is smaller than a specified control limit $ c$.
Hence, we define
$$
  S_T = S_T(c) = \inf \{ k \le t \le T : D_T(t/T) < c \}, \qquad \inf \emptyset = \infty.
$$
We will assume that the start of monitoring, $k$, is given by
$$
  k = \trunc{T \kappa}, \qquad \text{for some $ \kappa \in (0,1) $}.
$$
A reasonable approach to choose $ c $ is to control the type I error rate
$ \alpha \in (0,1) $, i.e.,
to ensure that
\begin{equation}
\label{AsLimitAlpha}
  \lim_{T \to \infty} P_0( S_T(c) \le T ) = \alpha,
\end{equation}
where $ P_0 $ indicates that the probability is calculated assuming that
$ \{ Y_t \} $ is a random walk corresponding to the null hypothesis $ H_0: \rho = 1 $.

\subsection{DF control chart with estimated control limit}

In the next section we will show that $ D_T $ converges weakly to some
stochastic process $ \calD_\vartheta $ depending on the nuisance parameter
$$
  \vartheta = \lim_{T \to \infty} \vartheta_T = \eta / \sigma,
$$
and that $ S_T/T $ converges in distribution
to $ \inf \{ s \in [\kappa,1] : \calD_\vartheta(s) < c \} $. Hence, if $ c $ is chosen
from the asymptotic distribution via (\ref{AsLimitAlpha}), $c = c(\vartheta) $
is a function of $ \vartheta $. Therefore, the basic idea is to estimate
$ \vartheta $ at each time point using only past and current data, and to
use the corresponding limit.

Our estimator for $ \vartheta $ will be based on a Newey-West type estimator,
thus circumventing the problem to specify the short memory dynamics of the process
explicitly.
Let $ \gamma(k) = E( \epsilon_t \epsilon_{t+k} ) $ and denote by
$ r(k) = \gamma(k) / E ( \epsilon_t^2 ) $,
$ k \in \N $, the autocorrelation function of the time series $ \{ \epsilon_t \} $.
Since $ \epsilon_t = \Delta Y_t $ if $ \rho = 1 $, we
can estimate $ \gamma(k) $ and $ r(k) $ under the null hypothesis by
\begin{equation}
\label{DefNW1}
  \widehat{r}_t(k) = \widehat{\gamma}_t(k) / \widehat{\sigma}^2_t,
\quad
  \widehat{\gamma}_t(k) = t^{-1} \sum_{s=k}^t \Delta Y_s \Delta Y_{s-k}, \quad
  \widehat{\sigma}^2_t = t^{-1} \sum_{s=1}^t \Delta Y_s^2.
\end{equation}
The parameter $ \vartheta^2 $ can now be estimated by the Newey-West estimator
%$ \lim_{t \to \infty} \Var( t^{-1/2} \sum_{s=1}^t \Delta Y_s )  $
given by
\begin{equation}
\label{DefNW2}
  \widehat{\vartheta}_t^2 = \widehat{\eta}^2_t / \widehat{\sigma}^2_t,
  \qquad \widehat{\eta}^2_t = \widehat{\sigma}^2_t + 2 \sum_{i=1}^m w(m,i) \widehat{\gamma}_t^2(i),
\end{equation}
where $ w(m,i) = (m-i)/m $ are the Bartlett weights and $m$ is a lag truncation
parameter, see Newey and West (1987). Andrews (1991) studies more general weighting functions
and shows that the rate $ m = o( T^{1/2} ) $ is sufficient for consistency.

The Dickey-Fuller control chart for correlated time series works now as follows.
At each time point $ t $ we estimate
 $ \vartheta $ by $ \widehat{\vartheta}_t $ and calculate
the corresponding estimated control limit $ c( \widehat{\vartheta}_t ) $. A signal
is given if $ D_T $ is less than the estimated control limit, i.e., we use the rule
$$
  \widehat{S}_T
  = \inf \{ k \le t \le T : D_T( t/T ) < c(\widehat{\vartheta}_t) \}.
$$

\subsection{DF control chart based on a transformation}

Alternatively, one may use a transformation of $ D_T $, namely
\begin{equation}
\label{DefTrans}
  E_T(s) = D_T(s)
    + \frac{
        \frac{ \widehat{\sigma}^2_{\trunc{Ts}} - \widehat{\eta}^2_{\trunc{Ts}} }{ 2 \trunc{Ts} }
               \sum_{t=1}^{\trunc{Ts}} K( (\trunc{Ts}-t)/h )
      }{
        \trunc{Ts}^{-2} \sum_{t=1}^{\trunc{Ts}} Y_{t-1}^2
      },
      \qquad s \in (0,1].
\end{equation}
It seems that this transformation idea dates back to Phillips (1987).
We will show that for arbitrary $ \vartheta $ the process
$ E_T $ converges weakly to the limit of $ D_T $ for $ \vartheta = 1 $.
Consequently, if $ c $ denotes the control limit ensuring that $ S_T $ has
size $ \alpha $ when $ \vartheta = 1 $, then the detection rule
$$
  Z_T = \inf\{ k \le t \le T : E_T( t/T ) < c \}
$$
has asymptotic size $ \alpha $ for any $ \vartheta $.

In the next section we shall show that both procedures are asymptotically valid.

\subsection{Extensions to Dickey-Fuller $t$-processes}

Inference on the AR parameter in the unit root case is often based on the
$ t $-statistic associated with $ D_T $, which gives rise to Dickey-Fuller $t$-processes.
The Dickey-Fuller $ t $-statistic, $ t_{DF} $, associated with
$ D_T = T( \widehat{\rho}_T - 1 ) $, is the standard computer output quantity when
running a regression of $ Y_t $ on $ Y_{t-1} $. For a sample
$ Y_1, \dots, Y_T $, the statistic $ t_{DF} $ is defined as
$$
  t_{DF} = (\widehat{\rho}_T -1) / \widehat{\xi}_T = T (\widehat{\rho}_T - 1) / (T \widehat{\xi}_T)
$$
where
$$
 \widehat{\xi}_T = \left\{ s_T^2 \Big/ \sum_{t=1}^{T} Y_{t-1}^2 \right\}^{1/2}
$$
with $ s_T^2 = (T-1)^{-1} \sum_{t=1}^T (Y_t - \widehat{\rho}_T Y_{t-1})^2 $.

The formula for $ t_{DF} $ motivates to scale $ D_T $ analogously.
Hence, let us define the weighted $t$-type DF process by
\begin{equation}
\label{DefDFT}
  \widetilde{D}_T(s) = D_T(s) / ( \trunc{Ts} \widehat{\xi}_{\trunc{Ts}} ), \qquad s \in (0,1],
\end{equation}
and $ \widetilde{D}_T(0) = 0 $. $ \widetilde{D}_T(s) $ is a weighted
version of $ t_{DF} $ calculated using the observations $ Y_1,
\dots, Y_{\trunc{Ts}} $, and attaching kernel weights $ K( (
\trunc{Ts} - t )/h ) $ to the $ t$th summand in the numerator. The
associated detection rule for known $ \vartheta $ is defined as
$$
  \widetilde{S}_T
    = \widetilde{S}_T(c) = \inf \{ k \le t \le T : \widetilde{D}_T( t/T ) < c(\vartheta) \}
$$
with $ c( \vartheta ) $ such that
$ \lim_{T \to \infty} P_0( \widetilde{S}_T( c( \vartheta ) )  \le T ) = \alpha $.

Again, it turns out that the asymptotic limit of $ \widetilde{D}_T $ depends
on the nuisance parameter $ \vartheta $.
The weighted $t$-type DF control chart with estimated control limits is defined as
$$
  \widehat{ \widetilde{S} }_T =
    \inf \{ k \le t \le T : \widetilde{D}_T( k/T ) < c( \widehat{\vartheta}_t ) \}.
$$

Alternatively, one can transform the process to achieve that the asymptotic limit
is invariant with respect to $ \vartheta $. We define
\begin{equation}
\label{DefTransT}
  \widetilde{E}_T(s) = \frac{ S_{\trunc{Ts}} }{ \widehat{\eta}_{\trunc{Ts}} }
  \widetilde{D}_T(s)
    -
  \frac{
     \frac{ \widehat{\eta}_{\trunc{Ts}}^2 - \widehat{\sigma}^2_{\trunc{Ts}} }{ 2 \trunc{Ts} }
     \sum_{t=1}^{\trunc{Ts}} K( (\trunc{Ts}-t)/h )
  }
  {
    \widehat{\eta}_{\trunc{Ts}} \sqrt{ \trunc{Ts}^{-2} \sum_{t=1}^{\trunc{Ts}} Y_{t-1}^2 }
  },
  \qquad s \in (0,1].
\end{equation}
We will show that the detection rule
$$
  \widetilde{Z}_T = \inf\{ k \le t \le T : \widetilde{E}_T( t/T ) < c(1) \}
$$
has asymptotic type I error equal to $ \alpha $ for all $ \vartheta $.

\section{Asymptotic results for random walks}
\label{AsH0}

In this section we provide functional central limit theorems for the
Dickey-Fuller processes defined in the previous section under a
random walk model assumption corresponding to the null hypothesis $
H_0 : \rho = 1 $ in model (\ref{ARModel}), and the related central
limit theorem for the associated stopping rules. These results can
be used to design tests and detection procedures having well-defined
statistical properties under the null hypothesis.

\subsection{Weighted Dickey-Fuller processes}

We start with the following functional central limit theorem providing the
limit distribution of the weighted DF process $ D_T(s)$, $ s \in [0,1] $, which
extends Phillips (1987, Th. 3.1 c).

\begin{theorem}
\label{ThBasic} Assume the time series  $ \{ Y_t \}$ satisfies model
(\ref{ARModel}) with $ \rho = 1 $ such that (E1) and (K1)-(K3) hold.
Then
$$
  D_T(s) \Rightarrow \calD_\vartheta(s), \qquad \mbox{in $(D[\kappa,1],d)$},
$$
as $ T \to \infty $, where the stochastic process
\begin{equation}
\label{DefD}
  \calD_\vartheta(s) =
    \frac{ \frac{s}{2} \left\{ K(0) B(s)^2 + \zeta \int_0^s B(r)^2 K'( \zeta(s-r) ) \, dr
           - \vartheta^{-2} \int_0^s K( \zeta( s-r ) ) \, dr \right\} }
    { \int_0^s B^2(r) \, dr },
\end{equation}
$ s \in (0,1] $, $ \calD_\vartheta(0)=0$, is continuous w.p. $1$.
\end{theorem}

\begin{remark}
  Note that the asymptotic limit is distribution-free if and only if
  $ \eta = \sigma $ which holds if the error terms are uncorrelated. Otherwise,
  the distribution of $ \calD_\vartheta $ depends sensitively on $ \vartheta $.
\end{remark}

\begin{proof} If $ \rho = 1 $ we have
 $ \epsilon_t = \Delta Y_t $ and
$ Y_{t-1} \epsilon_t = (1/2) (Y_t^2 - Y_{t-1}^2 - \epsilon_t^2) $
for all $ t $. This yields the representation
$$
  D_T(s) = \frac{ \widetilde{V}_T(s) - \widetilde{R}_T(s)}{ \widetilde{W}_T(s) },
  \qquad s \in (0,1],
$$
where the $ D[0,1] $-valued stochastic processes $ \widetilde{V}_T $,
$ \widetilde{R}_T $, and $ \widetilde{W}_T $  are given by
\begin{eqnarray*}
  \widetilde{V}_T(s) & = &
    (2\trunc{Ts})^{-1} \sum_{t=1}^{ \trunc{Ts} } ( Y_t^2 - Y_{t-1}^2 ) K( ( \trunc{Ts} - t)/h ), \\
  \widetilde{R}_T(s) & = &
    (2\trunc{Ts})^{-1} \sum_{t=1}^{ \trunc{Ts} } \epsilon_t^2 K( (\trunc{Ts}-t)/h ), \\
  \widetilde{W}_T(s)
    & = & \trunc{Ts}^{-2} \sum_{t=1}^{ \trunc{Ts} } Y_{t-1}^2
\end{eqnarray*}
for $ s \in (0,1] $.
Let us first show that
\begin{equation}
\label{Rem1}
  \sup_{s \in [\kappa,1]} | \widetilde{R}_T(s) - \mu(s) |
  \stackrel{P}{\to} 0,
\end{equation}
as $ T \to \infty $, where
$$
  \mu(s) = \frac{\sigma^2}{2 s} \int_0^s K( \zeta(s-r) ) \, dr, \qquad s \in (0,1].
$$
Consider
$$
  | E( \widetilde{R}_T(s) ) - \mu(s) |
    = \frac{\sigma^2}{2}
      \left| \frac{1}{\trunc{Ts}} \sum_{t=1}^{ \trunc{Ts} } K( (\trunc{Ts}-t)/h )
        - s^{-1} \int_0^s K(\zeta(s-r)) \, dr \right|.
$$
(\ref{AssZeta}) ensures that
$ \sup_{s \in [\kappa,1]} \max_i | (\trunc{Ts} - i)/h - \zeta(s-i/T) | = o(1) $ yielding
\begin{eqnarray*}
  \frac{1}{\trunc{Ts}} \sum_{t=1}^{\trunc{Ts}} K( ( \trunc{Ts} - t ) /h )
  &=& \frac{1}{\trunc{Ts}} \sum_{t=1}^{\trunc{Ts}} K( \zeta(s-t/T) ) + o(1) \\
  &=& s^{-1} \int_0^s K( \zeta(s-r) ) \, dr + o(1),
\end{eqnarray*}
uniformly in $ s \in [\kappa,1] $, because $K$ is Lipschitz
continuous and of bounded variation, cf. Theorem 3.3(ii) of Steland
(2004). It remains to estimate $ | \widetilde{R}_T(s) - E(
\widetilde{R}_T(s) ) | $. The assumptions on $ \{ \epsilon_t \} $
ensure that
$$
  Z_T(r) = T^{-1/2} \sum_{i=1}^{ \trunc{Tr} } ( \epsilon_i^2 - E \epsilon_i^2 )
  \Rightarrow \rho B^{(2)}(r)
$$
as $ T \to \infty $, where $ \rho^2 = \Var(\epsilon_1^2) + 2
\sum_{t=1}^\infty \Cov( \epsilon_1^2, \epsilon_{1+t}^2 ) $. Hence,
eventually for equivalent versions, we may assume that $ \| Z_T -
\rho B^{(2)} \|_\infty \to 0 $ a.s., for $ T \to \infty $. By (K3)
the Stieltjes integrals $ \int_0^s K(\zeta(s-r)) \, dB^{(2)}(r) $
and $ \int_0^s K(\zeta(s-r)) \, d Z_T(r) $ are well defined (via
integration by parts), and
$$
  \sup_{s \in [\kappa,1]} \left| \int_0^s K( \zeta(s-r) ) \, d Z_T(r)  - \rho \int_0^s K( \zeta(s-r) ) \, d B^{(2)}(r) \right|
  \stackrel{a.s.}{\to} 0,
$$
as $ T \to \infty $. Obviously,
\begin{eqnarray*}
  && \sup_{s \in [\kappa,1]} | \widetilde{R}_T(s) - E \widetilde{R}_T(s) |
    = \sup_{s \in [\kappa,1]} \frac{\sqrt{T}}{\trunc{Ts}} \left|
      \int_0^s K( ( \trunc{Ts} - \trunc{Tr} )/h ) d Z_T(r) \right| \\
    && \le \sup_{s \in [\kappa,1]} \frac{ \rho \sqrt{T}}{ \trunc{Ts} }
            \left| B^{(2)}(r) K\bigl[ \frac{ \trunc{Ts} - \trunc{Tr} }{h} \bigr] \bigr|_{r=0}^{r=s} - \int_0^s B^{(2)}(r) K( (\trunc{Ts}-\trunc{T(dr)})/h )  \right| \\
    && \quad + \sup_{s \in [\kappa,1]} \frac{ \sqrt{T} }{ \trunc{Ts} }
            \biggl| [ Z_T(r) - \rho B^{(2)}(r) ] K( ( \trunc{Ts} - \trunc{Tr})/h
            ) \bigr|_{r=0}^{r=s} \\
    && \qquad \quad - \int_0^s [Z_T(r) - \rho B^{(2)}(r)] K( ( \trunc{Ts} - \trunc{ T(dr) } ) / h )
    \biggr|.
\end{eqnarray*}
Noting that the total variation of the functions $ r \mapsto K(
[\trunc{Ts} - \trunc{Tr}]/h ) $, $ s \in [\kappa,1] $, $ T \ge 1 $,
is uniformly bounded, the right side of the above display can be
estimated by
\begin{eqnarray*}
    &&  O\left( \frac{\kappa^{-1}}{\sqrt{T}} \| B^{(2)} \|_\infty \| K \|_\infty \right) + O\left( \frac{\kappa^{-1}}{\sqrt{T}} \| B^{(2)} \|_\infty \int |d K|
    \right) \\
    && \quad + O\left( \frac{\kappa^{-1}}{\sqrt{T}} \| Z_T - \rho B^{(2)} \|_\infty \| K \|_\infty \right) + O\left( \frac{\kappa^{-1}}{\sqrt{T}} \| Z_T - \rho B^{(2)} \|_\infty \int |d K| \right) \\
    && = O_P(1/\sqrt{T}) = o_P(1).
\end{eqnarray*}
Therefore,  (\ref{Rem1}) holds true. Let us now consider $
\widetilde{V}_T $. We will first show that, up to terms of order $
o_P(1) $, $ \widetilde{V}_T $ is a functional of
$$
  U_T(r) = T^{-1/2} Y_{ \trunc{Tr} }, \qquad r \in [0,1].
$$
Again, under the assumptions of the theorem, $ U_T $  converges weakly to $ \eta B $, where
$ B $ denotes Brownian motion and $ \eta > 0 $ is a constant. For brevity of
notation let
$$
  k_T(r;s) = K( (\trunc{Ts} - \trunc{Tr})/ h ), \qquad r,s \in [0,1].
$$
Integration by parts yields
\begin{eqnarray*}
  \widetilde{V}_T(s) &=&
    \frac{1}{2\trunc{Ts}}
      \sum_{t=1}^{ \trunc{Ts} } ( Y_t^2 - Y_{t-1}^2 ) K( ( \trunc{Ts} - t)/h ) \\
    &=& \frac{T}{ 2 \trunc{Ts} }
          \int_0^s k_T(r;s) \, d (T^{-1/2} Y_{\trunc{Tr}})^2 \\
    &=& \frac{T}{2 \trunc{Ts}} \left(
          k_T(r;s) U_T^2(r) \biggr|_{r=0}^{r=s}
          - \int_0^s U_T^2(r) \, k_T( dr;s)
        \right) \\
    &=& \frac{ K(\zeta(s-r)) }{ 2s } U_T^2(r) \biggr|_{r=0}^{r=s}
       + \frac{\zeta}{2s} \int_0^s U_T^2(r) K'( \zeta(s-r) ) \, dr + o_P(1) \\
     &=& \frac{ \eta^2 K(0) B^2(s) }{ 2 s }
       + \frac{ \zeta }{ 2s } \int_0^s U_T^2(r) K'( \zeta (s-r) ) \, dr + o_P(1).
\end{eqnarray*}
Due to (K2) the $ o_P(1) $ term is uniform in $ s \in [\kappa,1] $.
Next note that
$$
  \widetilde{W}_T(s) = \left( \frac{T}{\trunc{Ts}} \right)^2
    \int_0^s U_T^2(r) \, dr.
$$
We are now in a position to verify joint weak convergence of numerator and
denominator of $ D_T $.
The Lipschitz continuity of $K$ ensures that up to terms of order $ o_P(1) $
for all $ (\lambda_1,\lambda_2) \in \R^2 $ the linear combination
$
  \lambda_1 ( \widetilde{V}_T(s) - \widetilde{R}_T(s) ) + \lambda_2 \widetilde{W}_T(s)
$
is a functional of $ U_T $, and that functional is continuous. Therefore, the
continuous mapping theorem (CMT) entails
weak convergence to the stochastic process
\begin{eqnarray*}
 &&   \lambda_1 \left[ \frac{ \eta^2 K(0) B^2(s) }{2s}
    + \frac{\eta^2 \zeta}{2s} \int_0^s K'(\zeta(s-r)) B^2(r) \, dr
    - \frac{\sigma^2}{2s} \int_0^s K(\zeta(s-r)) \, dr
    \right] \\
 && \quad
    + \lambda_2 \frac{\eta^2}{s^2} \int_0^s B(r)^2 \, dr.
\end{eqnarray*}
This verifies joint weak convergence of
$ (\widetilde{V}_T - \widetilde{R}_T, \widetilde{W}_T) $.
Hence, the result follows by the CMT. (K2) also ensures that $ \calD_\vartheta \in C[0,1] $ w.p. $1$.
\end{proof}

The central limit theorem (CLT) for the detection procedure $ S_T $, which requires knowledge of $ \vartheta $,
appears as a corollary.

\begin{corollary}
\label{CorST}
Under the assumptions of Theorem~\ref{ThBasic} we have for any
control limit $ c < 0 $
$$
  S_T / T \stackrel{d}{\to} \inf \{ s \in [\kappa,1] : \calD_\vartheta(s) < c \}
$$
as $ T \to \infty $, where $ \calD_\vartheta(s) $ is defined in (\ref{DefD}).
\end{corollary}

\begin{proof} Observe that by definition of $ S_T $
$$
  S_T > x \Leftrightarrow \inf_{s \in [\kappa,x]} D_T(s) \ge c
          \Leftrightarrow \sup_{s \in [\kappa,x]} -D_T(s) \le -c
$$
for any $ x \in \R $. Hence it suffices to show that
$$
  P( \sup_{s \in [\kappa,x]} -D_T(s) \le -c )
    \to P( \sup_{s \in [\kappa,x]} -\calD_\vartheta(s) \le -c ),
$$
where $ \calD_\vartheta $ denotes the limit process given in Theorem~\ref{ThBasic}.
Using the Skorokhod/Dudley/ Wichura representation theorem and a result
due to Lifshits (1982),
this fact can be shown along the lines of the proof of Theorem~4.1 in
Steland (2004), if $ c < 0 $, since $ \calD_\vartheta \in C[0,1] $ a.s.
For brevity we omit the details.
\end{proof}

Let us now show consistency of the detection procedure
$ \widehat{S}_T = \inf \{ k \le t \le T : D_T(t/T) < c( \widehat{\vartheta}_t ) \} $,
which uses estimated control limits.

\begin{theorem}
\label{ThEstCL}
  Assume (E1) and (E2), (K1)-(K3), and in addition that the lag truncation parameter, $m$,
  of the Newey-West estimator satisfies
  $$
    m = o( T^{1/2} ), \qquad T \to \infty.
  $$
  Then the weighted Dickey-Fuller type control chart with estimated control limit,
  $ \widehat{S}_T $, is consistent, i.e.,
  $$
    P( \widehat{S}_T \le T ) \to \alpha,
  $$
  as $ T \to \infty $.
\end{theorem}

\begin{proof}
  Note that the equivalence $ \widehat{S}_T > T \Leftrightarrow
  \inf_{s \in [\kappa,1]} D_T(s) / c( \widehat{\vartheta}_{\trunc{Ts}} ) \ge 1  $
  implies
  \begin{equation}
  \label{Stern}
  P( \widehat{S}_T \le T )
  =
  P\left(
    \inf_{s \in [\kappa,1]}
      \frac{D_T(s)}{c( \widehat{\vartheta}_{\trunc{Ts}} )} < 1
  \right).
  \end{equation}
  Let us first show that the function $c$ is continuous. Note that
  the process $ \calD_\vartheta(s) $ can be written as
  $ \calE(s) - \vartheta^{-2} \calF(s) $ for
  a.s. continuous processes $ \calE $ and $ \calF $ not depending on $ \vartheta $,
  where particularly $ \calF(0) = 0 $ and
  $$
    \calF(s) = (s/2) \int_0^s K( \zeta( s-r ) ) \, d r / \int_0^s B^2(r) \, dr,
    \qquad s \in (0,1].
  $$
  Let $ \{ \vartheta^*, \vartheta_n : n \ge 1 \} \subset \R $ be a sequence with
  $ \vartheta_n \to \vartheta^* $, as $ n \to \infty $. Clearly,
  for each $ \omega $ of the underlying probability space with
  $ | \calE(\omega) |, | \calF(\omega) | < \infty $, we have
  $$
  \calD_{\vartheta_n}(\omega) = \calE(\omega) + \vartheta_n^{-2} \calF(\omega) \to
    \calE(\omega) + (\vartheta^*)^{-2} \calF(\omega) = \calD_{\vartheta^*}(\omega), $$ $ n \to \infty $.
  Hence,
  $ \sup_{s \in [\kappa,1]} \calD_{\vartheta_n}(s) \stackrel{d}{\to} \sup_{s \in [\kappa,1]} \calD_{\vartheta^*}(s) $,
  as $ n \to \infty $. Since $ \sup_{s \in [\kappa,1]} \calD_{\vartheta^*}(s) $ has a continuous density,
  this is equivalent to pointwise convergence of the d.f.
  $ F_n(z) = P( \sup_{s \in [\kappa,1]} \calD_{\vartheta_n}(s) \le z ) $ to
  $ F(z) = P( \sup_{s \in [\kappa,1]} \calD_{\vartheta^*}(s) \le z )$, as $ n \to \infty $, for all $ z \in \R $.
  Hence,
  $$
    c( \vartheta_n ) = F_n^{-1}(\alpha) \to F^{-1}(\alpha) = c( \vartheta^* ),
  $$
  as $ n \to \infty $. Next we show
  \begin{equation}
  \label{WCTheta}
    \widehat{\vartheta}_{\trunc{Ts}} \Rightarrow \vartheta,
  \end{equation}
  as $ T \to \infty $, in $ D[\kappa,1] $.
  Since for each $ s \in [\kappa,1] $ we have $ \widehat{\vartheta}_{\trunc{Ts}} \stackrel{P}{\to} \vartheta $, for $ T \to \infty $,
  fidi convergence follows immediately.
  It remains to verify tightness. Recall the definitions
  (\ref{DefNW1}) and (\ref{DefNW2}) and that $ \Delta Y_t = \epsilon_t $ under $ H_0 $.
  Fix $j$ and consider the process $ \widehat{\gamma}_{\trunc{Ts}}(j) $, $ s \in [\kappa,1] $,
  which is a functional of $ \{ \epsilon_t \epsilon_{t-j} : t = j, j+1,\dots \} $.
  Clearly, by the Cauchy-Schwarz inequality and (E1)
  $ E | \epsilon_t \epsilon_{t-j} |^{2+\delta} \le E | \epsilon_t |^{4+2\delta} < \infty $
  for some $ \delta > 0 $.
  Further, since
  $ \widetilde{\calF}_{-\infty}^t = \sigma( \epsilon_s \epsilon_{s-j} : s \le t ) \subset
  \calF_{-\infty}^t  = \sigma( \epsilon_s : s \le t ) $ and
  $ \widetilde{\calF}_{t+k}^\infty = \sigma( \epsilon_s \epsilon_{s-j} : s \ge t + k ) \subset
    \calF_{t+k}^\infty = \sigma( \epsilon_s : s \ge t + k -j ) $,
  the mixing coefficients $ \widetilde{\alpha}(k) $ of $ \{ \epsilon_t \epsilon_{t-j} \} $ satisfy
  \begin{eqnarray*}
    \widetilde{\alpha}(k) &=& \sup_t \sup_{A \in \widetilde{\calF}_{-\infty}^t, B \in \widetilde{\calF}_{t+k}^\infty}
    | P(A \cap B) - P(A)P(B) |  \\
    &\le&
    \sup_t \sup_{A \in \calF_{-\infty}^t, B \in \calF_{t+k-j}^\infty}
    | P(A \cap B) - P(A)P(B) | = \alpha(k-j),
  \end{eqnarray*}
  where $ \{ \alpha(k) \} $ are the mixing coefficients of $ \{ \epsilon_t \} $.
  Due to (E1) we can apply Yokohama (1980,~Th.1) with $ r = 2 + 2 \delta $ to conclude
  that for $ \kappa \le r \le s \le 1 $
  $$
    E \left| T^{-1/2} \sum_{t=\trunc{Tr}+1}^{\trunc{Ts}}
      \epsilon_t \epsilon_{t-j} \right|^{2+2\delta} = O( |s-r|^{1+\delta} ).
  $$
  Now the decomposition
  $$
    \sqrt{T}( \widehat{\gamma}_{\trunc{Ts}}(j) - \widehat{\gamma}_{\trunc{Tr}}(j) )
    = \frac{T}{\trunc{Ts}} \frac{1}{\sqrt{T}} \sum_{t=\trunc{Tr}+1}^{\trunc{Ts}} \epsilon_t \epsilon_{t-j}
    + \left( \frac{T}{\trunc{Ts}} - \frac{T}{\trunc{Tr}} \right)
      \frac{1}{\sqrt{T}} \sum_{t=k}^{\trunc{Tr}} \epsilon_t \epsilon_{t-j}
  $$
  and the triangle inequality yield
  \begin{eqnarray*}
    \| \sqrt{T}( \widehat{\gamma}_{\trunc{Ts}}(j) - \widehat{\gamma}_{\trunc{Tr}}(j) )
    \|_{2+2\delta}
    &=& O( s^{-1} |s-r|^{(1+\delta)/(2+2\delta)} ) + O( |1/s - 1/r|
    r^{(1+\delta)/(2+2\delta)} ) \\
    &=& O( |s-r|^{(1+\delta)/(2+2\delta)} ),
  \end{eqnarray*}
  since, firstly, we may assume $ 0 < \delta < 1 $, and, secondly, both $ s^{-1} $ and
  $ r^{(-1-\delta)/(2+2\delta)} $ are bounded away from $0$ and $ \infty $ for $ 0
  \le r \le s \le 1 $.
  Consequently,
  $$ E( \sqrt{T}( \widehat{\gamma}_{\trunc{Ts}}(j) - \widehat{\gamma}_{\trunc{Tr}}(j) ) )^{2+2\delta} = O( |s-r|^{1+\delta} ), $$
  and therefore Vaart and Wellner (1986, Ex.~2.2.3) implies tightness of the process $ \{ \sqrt{T} \widehat{\gamma}_{\trunc{T\cdot}}(j) : s \in [\kappa,1] \} $ for fixed $j \ge 0
  $. Note that $ \widehat{\gamma}_{\trunc{Ts}}(0) =
\sigma_{\trunc{Ts}}^2 $.
  By the triangle inequality we have
  \begin{eqnarray*}
    \| \sqrt{T}( \widehat{\eta}_{\trunc{Ts}} - \widehat{\eta}_{\trunc{Tr}} ) \|_{2+2\delta}
    &\le&
    2 \sum_{j=0}^m (1-j/m)^{2+2\delta}
      \| \widehat{\gamma}_{\trunc{Ts}}(j) - \widehat{\gamma}_{\trunc{Tr}}(j) \|_{2+2\delta} \\
    &=& O( m |s-r|^{(1+\delta )/(2+\delta )} ),
  \end{eqnarray*}
  yielding
  $$
    E | \widehat{\eta}_{\trunc{Ts}} - \widehat{\eta}_{\trunc{Tr}} |^{2+2\delta}
    = O( (m/T^{1/2})^{2+2\delta} |s-r|^{1+\delta} ).
  $$
  Hence,
  $ \{ (\widehat{\eta}_{\trunc{Ts}}, \widehat{\sigma}_{\trunc{Ts}}^2) : s \in [\kappa,1] \} $
  is tight in the product space, which implies weak convergence of
  $ \{ \widehat{\vartheta}_{\trunc{Ts}} : s \in [\kappa,1] \} $ to $
  \vartheta $.
  The final step is to verify
  \begin{equation}
  \label{WC2}
    \inf_{s \in [\kappa,1]} D_T(s) / c( \widehat{\vartheta}_{\trunc{Ts}} )
    \stackrel{d}{\to} \inf_{s \in [\kappa,1]} \calD_\vartheta(s) / c( \vartheta ),
  \end{equation}
  as $ T \to \infty $, since this implies that (\ref{Stern}) converges to
  $ P( \inf_{s \in [\kappa,1]} \calD_\vartheta(s) < c(\vartheta) ) = \alpha $,
  as $ T \to \infty $.
  Due to (\ref{WCTheta}) we can conclude that
  $$
    ( D_T(\cdot), \widehat{\vartheta}_{\trunc{T \cdot}} )
      \Rightarrow ( \calD_\vartheta(\cdot), \vartheta )
  $$
  in the product space $ ( D[\kappa,1] )^2 $.  Note that the mapping
  $ \varphi : ( D[\kappa,1],d )^2 \to (\R, \calB) $ given by
  $$
    \varphi(x,y) = \inf_{s \in [\kappa,1]} \frac{x(s)}{c(y(s))}, \qquad x,y \in D[\kappa,1], \ y \in \R,
  $$
  is continuous in all $ (x,y) \in (C[\kappa,1])^2 $. Since $ D_\vartheta \in C[0,1] $
  w.p. $1$ and $c \in C(\R) $,  (\ref{WC2}) follows.
\end{proof}

It remains to provide the related weak convergence results for the
transformed process $ E_T $ and its natural detection rule
$ Z_T = \inf\{ k \le t \le T : E_T(t/T) < c \} $.

\begin{theorem}
  Assume (E1),(E2), and (K1)-(K3). Additionally assume that the lag truncation parameter, $m$,
  of the Newey-West estimator satisfies
  $$
    m = o( T^{1/2} ), \qquad T \to \infty.
  $$
  Then,
  $$
    E_T(s) \Rightarrow \calD_1(s), \qquad \mbox{in $ (D[\kappa ,1],d)$}
  $$
  as $ T \to \infty $, and for the transformed Dickey-Fuller type control chart we have
  $$
    Z_T / T \stackrel{d}{\to} \inf\{ \kappa \le t \le 1 : \calD_1(t) < c \}.
  $$
  as $ T \to \infty $.
  Particularly, the asymptotic distributions are invariant with respect to $ \vartheta $.
\end{theorem}

\begin{proof} As shown above,
$$
 \widehat{\eta}^2_{\trunc{T \cdot}} \Rightarrow \eta^2
 \quad \text{and} \quad
 \widehat{\sigma}^2_{\trunc{T \cdot}} \Rightarrow \sigma^2,
$$
as $ T \to \infty $, which implies that
$$
  \left( D_T(s), \trunc{Ts}^{-2} \sum_{t=1}^{\trunc{Ts}} Y_{t-1}^2,
\widehat{\eta}^2_{\trunc{Ts}}, \widehat{\sigma}^2_{\trunc{Ts}} \right)
 \Rightarrow ( \calD_\vartheta(s), \eta^2/s^2 \int_0^s B^2(r) \, dr, \eta^2, \sigma^2 ),
$$
if $ T \to \infty $, yielding
\begin{eqnarray*}
 && D_T(s)
    + \frac{
        \frac{ \widehat{\sigma}^2_{\trunc{Ts}} - \widehat{\eta}^2_{\trunc{Ts}} }{ 2 }
               \frac{1}{\trunc{Ts}} \sum_{t=1}^{\trunc{Ts}} K( (\trunc{Ts}-t)/h )
      }{
        \trunc{Ts}^{-2} \sum_{t=1}^{\trunc{Ts}} Y_{t-1}^2
      } \\
 && \qquad
  \Rightarrow \calD_\vartheta(s) + \frac{ \sigma^2 - \eta^2 }{ 2 \eta^2 }
  \frac{ s^{-1} \int_0^s K(\zeta(s-r)) \, dr }{ s^{-2} \int_0^s B^2(r) \, dr }
  = \calD_1(s).
\end{eqnarray*}
\end{proof}

\subsection{Weighted Dickey-Fuller $t$-processes}

Let us now derive (functional) central limit theorems for the weighted Dickey-Fuller
$t$-processes and the associated detection rules. We start with the process
$ \widetilde{D}_T $ under the random walk null hypothesis.

\begin{theorem}
\label{ThTType} Assume (E1), and (K1)-(K3). Then
$$
  \widetilde{D}_T \Rightarrow \widetilde{\calD}_\vartheta, \qquad \mbox{in $(D[\kappa,1],d)$}
$$
as $ T \to \infty $, where
$$ \widetilde{\calD}_\vartheta(s) =
   \frac{ \frac{1}{2} \left\{
           \vartheta K(0) B(s)^2 + \vartheta \zeta \int_0^s B(r)^2 K'(\zeta(s-r)) \, dr
            - \vartheta^{-1} \int_0^s K( \zeta(s-r) ) \, dr
         \right\} }
        { \left\{ \int_0^s B(r)^2 \, dr \right\}^{1/2} }
$$
for $ s \in (0,1] $ and $ \widetilde{\calD}_\vartheta(0) = 0 $. Here $ \vartheta = \eta/\sigma $.
$ \widetilde{\calD}_\vartheta $ is continuous a.s.
\end{theorem}

\begin{remark} Note that again the limit depends on the nuisance parameter $ \vartheta $
and is distribution-free if and only if $ \vartheta = 1 $.
\end{remark}

\begin{proof} By definition
  $$
    \widetilde{D}_T(s) = \frac{ D_T(s) }{ \trunc{Ts} \widehat{\xi}_{ \trunc{Ts} } }
  $$
  where
  $$
    \trunc{ Ts } \widehat{\xi}_{ \trunc{Ts} }
    =
    \sqrt{ \frac{ S^2_{ \trunc{Ts} } }{ \trunc{Ts}^{-2} \sum_{t=1}^{ \trunc{Ts} } Y_{t-1}^2 } }
  $$
  with
  $$
    S^2_{ \trunc{Ts} } =
      \frac{1}{ \trunc{Ts} - 1} \sum_{t=1}^{ \trunc{Ts} }
        \widehat{\epsilon}_t( \trunc{Ts} )^2, \qquad
    \widehat{\epsilon}_t( \trunc{Ts} ) = Y_t - \widehat{\rho}_{\trunc{Ts}} Y_{t-1},
  $$
  for $ s \in (0,1] $.
  Note that for $ t = 1, \dots, \trunc{Ts} $
  $$
   \widehat{\epsilon}_t( \trunc{Ts} ) - \epsilon_t = -( \widehat{\rho}_{\trunc{Ts}} - 1) Y_{t-1}.
  $$
  Hence, we obtain
  \begin{eqnarray*}
    S_{\trunc{Ts}}^2
      &=& \frac{1}{\trunc{Ts}-1} \sum_{t=1}^{\trunc{Ts}}
      (\epsilon_t + \{ \widehat{\epsilon}_t( \trunc{Ts} ) - \epsilon_t \})^2 \\
      &=& \frac{1}{\trunc{Ts}-1} \sum_{t=1}^{\trunc{Ts}} \epsilon_t^2
      - (\widehat{\rho}_{\trunc{Ts}} - 1)
        \frac{2}{\trunc{Ts}-1} \sum_{t=1}^{\trunc{Ts}} \epsilon_t Y_{t-1}
      + (\widehat{\rho}_{\trunc{Ts}} - 1)^2 \frac{1}{\trunc{Ts}-1} \sum_{t=1}^{\trunc{Ts}} Y_{t-1}^2.
  \end{eqnarray*}
  From the proof of Theorem~\ref{ThBasic} we know that
  $$
    \sup_{s \in (0,1]} \trunc{Ts}^{-2} \sum_{t=1}^{\trunc{Ts}} Y_{t-1}^2
    = \sup_{s \in (0,1]} \left( \frac{ T }{ \trunc{Ts} } \right)^2
        \int_0^s (T^{-1/2} Y_{\trunc{Tr}} )^2 \, dr = O_P(1)
  $$
  and
  $$
   \sup_{s \in (0,1]} \left| \trunc{Ts}^{-1} \sum_{t=1}^{\trunc{Ts}} \epsilon_t Y_{t-1} \right|
   \le \sup_{s \in (0,1]} \left| \trunc{Ts}^{-1/2} \sum_{t=1}^{\trunc{Ts}} \epsilon_t \right|
   \sup_{s \in (0,1]} | \trunc{Ts}^{-1/2} Y_{\trunc{Ts}} | = O_P(1).
  $$
  Combining these facts with
    $ \sup_{s \in (0,1]} \trunc{Ts} | \widehat{\rho}_{\trunc{Ts}} - 1 | = O_P(1) $,
  we obtain
  $$
    S_{\trunc{T s}}^2 = (\trunc{Ts}-1)^{-1} \sum_{t=1}^{\trunc{Ts}} \epsilon_t^2 + o_P(1),
  $$
  where the $ o_P(1) $ term is uniform in $ s \in (0,1] $.
  Because (E1) implies that
  $$
    \gamma_2(k) = \Cov( \epsilon_1^2, \epsilon_{1+k}^2 ) = o(1), \qquad |k| \to \infty,
  $$
  we may apply the law of large numbers for time series
  (Brockwell and Davis (1991), Th. 7.1.1) and obtain, since stochastic
  convergence to a constant yields stochastic convergence in the Skorokhod topology,
  \begin{equation}
  \label{ConvConst}
    d( S_{\trunc{T \circ }}^2, \sigma^2 ) \stackrel{P}{\to} 0,
  \end{equation}
  as $ T \to \infty $.
  We shall now show joint weak convergence of
  $ ( D_T(s), S^2_{\trunc{Ts}}, \trunc{Ts}^{-2} \sum_{t=1}^{\trunc{Ts}} Y_{t-1}^2 ) $,
  $ s \in (0,1] $.
  Let $ (\lambda_1,\lambda_2,\lambda_3) \in \R^3 - \{ \vecnull \}  $ and consider
  $$
    \lambda_1 D_T(s) + \lambda_2 S_{\trunc{T s}}^2
    + \lambda_3 \trunc{Ts}^{-2} \sum_{t=1}^{\trunc{T s}} Y_{t-1}^2,
    \qquad s \in [\kappa,1].
  $$
  The proof of Theorem~\ref{ThBasic} implies that
  $$
    \lambda_1 D_T(s) + \lambda_3 \trunc{Ts}^{-2} \sum_{t=1}^{\trunc{Ts}} Y_{t-1}^2
    \Rightarrow
    \lambda_1 \calD_\vartheta(s) + \lambda_3 \frac{\eta^2}{s^2} \int_0^s B(r)^2 \, dr,
  $$
  as $ T \to \infty $. Due to (\ref{ConvConst}), we obtain
  $$
    \lambda_1 D_T(s) + \lambda_2 S_{\trunc{T s}}^2
    + \lambda_3 \trunc{Ts}^{-2} \sum_{t=1}^{\trunc{T s}} Y_{t-1}^2
    \Rightarrow \lambda_1 \calD_\vartheta(s) + \lambda_2 \sigma^2
     + \lambda_3 \frac{\eta^2}{s^2} \int_0^s B(r)^2 \, dr,
  $$
  as $ T \to \infty $. Therefore, the CMT implies that
  $$
    \trunc{Ts} \widehat{\xi}_{\trunc{Ts}}
      \Rightarrow  \sqrt{  \frac{\sigma^2}{ \frac{ \eta^2 }{ s^2 } \int_0^s B^2(r) \, dr } }
      = \frac{ s }{ \vartheta \sqrt{ \int_0^s B^2(r) \,dr } }
  $$
  and
  $$
    \widetilde{D}_T(s) = \frac{ D_T(s) }{ \trunc{Ts} \widehat{\xi}_{\trunc{Ts}} }
     \Rightarrow  \frac{ \calD_\vartheta(s) \vartheta \sqrt{ \int_0^s B^2(r) \, dr } }{ s }
     = \widetilde{\calD}_\vartheta(s),
  $$
  as $ T \to \infty $,  yielding the assertion.
\end{proof}

We are now in the position to establish consistency of the $t$-type  detection rule
$$ \widehat{ \widetilde{S} }_T = \inf \{ k \le t \le T : \widetilde{D}_T( t/T) <
c( \widehat{\vartheta}_t ) \}, $$ which uses estimated control
limits. Notice that Theorem~\ref{ThTType} implies that $ c(
\vartheta ) $ is given by $ P_0( \inf_{s \in [\kappa,1]}
\widetilde{\calD}_\vartheta(s) < c(\vartheta ) ) = \alpha $.

\begin{theorem} Assume (E1),(E2), (K1)-(K3), and additionally
that the lag truncation parameter of the Newey-West estimator satisfies
$$
  m = o( T^{1/2} ), \qquad T \to \infty.
$$
Then the $t$-type weighted Dickey-Fuller control chart with estimated control limits,
$ \widehat{ \widetilde{S} }_T $, is consistent, i.e.,
$$
  P( \widehat{ \widetilde{S} }_T \le T ) \to \alpha,
$$
as $ T \to \infty $.
\end{theorem}

\begin{proof} The result is shown along the lines of the proof of Theorem~\ref{ThEstCL},
  since the process $ \widetilde{\calD}_\vartheta $ is continuous w.p. $1$, and
  is a continuous function of $ \vartheta $.
\end{proof}

Finally, for the transformed process $ \widetilde{E}_T $ and the
associated control chart $ \widetilde{Z}_T $ we have the following
result.

\begin{theorem}
Assume (E1),(E2), (K1)-(K3), and
$$
  m = o( T^{1/2} ), \qquad T \to \infty.
$$
Then the transformed $t$-type weighted DF process $ \widetilde{E}_T $, defined
in (\ref{DefTransT}), converges weakly,
$$
  \widetilde{E}_T \Rightarrow \widetilde{\calD}_1, \qquad \text{in $(D[0,1],d)$},
$$
as $ T \to \infty $, and for the transformed $t$-type weighted DF control chart we have
$$
  \widetilde{Z}_T / T \stackrel{d}{\to} \inf \{ \kappa < s < 1 : \calD_1(s) < c \}.
$$
Particularly, the asymptotic distribution is invariant with respect to $ \vartheta $.
\end{theorem}

\begin{proof} Note that the first term of $ \widetilde{E}_T $ converges weakly to
$ \vartheta^{-1} \widetilde{D}_\vartheta $, which has the form
$ [A(s) - \vartheta^{-2} \int_0^s K(\zeta(s-r)) \, dr]/ [ \int_0^s B^2(r) \, dr ]^{1/2} $.
Hence, the construction of the correction term is as for $ E_T $.
\end{proof}

\section{Asymptotics under local-to-unity alternatives}
\label{AsH1}

In econometric applications, the stationary alternatives
of interest are often of the form $ 0 < \rho < 1 $ with $ 1-\rho $ small.
To mimic this situation asymptotically, we consider a local-to-unity model
where the AR parameter depends on $ T $ and tends to $ 1 $,
as the time horizon $ T $ increases.

The functional central limit theorem given below shows that the
asymptotic distribution under local-to-unity alternatives is also
affected by the nuisance parameter $ \vartheta $. However, the term
which depends on the parameter parameterising the local alternative
does not depend on $ \vartheta $ (or $ \eta $). Therefore, if one
takes the nuisance parameter $ \vartheta $ into account when
designing a detection procedure, we obtain local asymptotic power.

Let us assume that we are given an array
$ \{ Y_{T,t} \} = \{ Y_{T,t} : 1 \le t \le T, T \in \N \} $ of observations
satisfying
\begin{equation}
\label{LocalToUnity}
  Y_{T,0} = 0, \quad Y_{T,t} = \rho_T Y_{T,t-1} + \epsilon_t, \qquad t = 1, \dots, T, \ T \ge 1,
\end{equation}
where the sequence of AR parameters $ \{ \rho_T \} $ is given by
$$
  \rho_T = 1 + a/T, \qquad T \ge 1,
$$
for some constant $ a $. $ \{ \epsilon_t \} $ is a mean-zero
stationary I(0) process satisfying (E1). For brevity of notation $
D_T $ denotes in this section the process (\ref{DefDT}) with $ Y_t $
replaced by $ Y_{T,t} $.

The limit distribution will be driven by an Ornstein-Uhlenbeck process.
Recall that the Ornstein-Uhlenbeck process $ \calZ_a $ with parameter $a$ is
defined by
\begin{equation}
\label{DefOU}
  \calZ_a(s) = \int_0^s e^{a(s-r)} \, d B(r), \qquad s \in [0,1],
\end{equation}
where $B$ denotes Brownian motion.

\begin{theorem}
\label{ThLocalUnity} Assume (E1),  and (K1)-(K3). Under the
local-to-unity model (\ref{LocalToUnity}) we have for the weighted
Dickey-Fuller process
$$
  D_T(s) \Rightarrow \calD_\vartheta^a(s),
$$
as $ T \to \infty $, where the a.s. $ C[0,1] $-valued process $ \calD_\vartheta^a $ is given by
$$
      \frac{ K(0) \calZ_a^2(s) + \zeta \int_0^s \calZ_a^2(r) K'( \zeta(s-r) ) \, dr - 2a \int_0^s \calZ^2_a(r) K( \zeta(s-r) ) \, dr
           - \frac{1}{\vartheta^2} \int_0^s K( \zeta( s-r ) ) \, dr }
    { (2/s) \int_0^s \calZ_a^2(r) \, dr }
$$
for $ s \in (0,1] $, and $ \calD_\vartheta^a(0) = 0$. Here $ \calZ_a
$ denotes the Ornstein-Uhlenbeck process defined in (\ref{DefOU}).
Further,
$$
  S_T/T \stackrel{d}{\to} \inf \{ s \in [\kappa,1] : \calD_\vartheta^a(s) < c \}, \qquad
  \mbox{as $ T \to \infty $}.
$$
\end{theorem}

\begin{proof} The crucial arguments to obtain {\em joint} weak convergence
of numerator and denominator of $ U_T $ have been given in detail in the proof of
Theorem~\ref{ThBasic}. Therefore, we give only a sketch of the proof stressing
the essential differences.
First, note that
$$
  U_T(s) = T^{-1/2} Y_{T,\trunc{Ts}} = \int_0^s e_T(r;s) \, d S_T(r), \qquad
  S_T(r) = T^{-1/2} \sum_{t=1}^{\trunc{Tr}} \epsilon_t,
$$
for the step function $ e_T(r;s) = (1+a/T)^{\trunc{Tr}-\trunc{Ts}} $, $ r,s \in [0,1] $,
which has uniformly bounded variation and
converges uniformly in $ r,s $ to the exponential $ e(r;s) = e^{a(s-r)} $. Hence,
firstly, the stochastic Stieltjes integral $ \int_0^s e_T(r;s) \, d S_T(r) $ exists
(via integration by parts), and, secondly, by estimating the terms of the decomposition
$ \int_0^s e_T \, d S_T - \int_0^s e d (\eta B) = \int_0^s (e_T-e) \, d (\eta B) + \int_0^s e_T \, d(S_T - \eta B) $
we see that
$$
  U_T(s) = \int_0^s e_T(r;s) \, d S_T(r) \Rightarrow \eta \int_0^s e(r;s) \, d B(r)
         = \eta \calZ_a(s),
$$
as $ T \to \infty $.
Next, note that in the local-to-unity model we have
$$
  Y_{T,t-1} \epsilon_t
    = \frac{1}{2 \rho_T}
        ( Y_{T,t}^2 - Y_{T,t-1}^2  +(1- \rho_T^2) Y_{T,t-1}^2 - \epsilon_t^2 )
$$
for all $ 1 \le t \le T $. This yields the decomposition
$$
  D_T(s) = \sum_{i=1}^3 \widetilde{V}_{i,T}(s) \Big/
  \widetilde{W}_T(s)
$$
where for $ s \in (0,1] $
\begin{eqnarray*}
  \widetilde{V}_{1,T}(s)
    & = & \frac{1}{2 \rho_T \trunc{Ts}} \sum_{t=1}^{\trunc{Ts}}
           (Y_{T,t}^2 - Y_{T,t-1}^2) K( ( \trunc{Ts}-t ) / h), \\
  \widetilde{V}_{2,T}(s)
    & = & \frac{1-\rho_T^2}{2 \rho_T \trunc{Ts}}
            \sum_{t=1}^{\trunc{Ts}} Y_{T, t-1}^2 K( ( \trunc{Ts}-t ) / h ), \\
  \widetilde{V}_{3,T}(s)
    & = & - \frac{1}{2\rho_T \trunc{Ts}} \sum_{t=1}^{\trunc{Ts}}
              \epsilon_t^2 K( (    \trunc{Ts}-t) / h ), \\
  \widetilde{W}_T(s)
    & = & \frac{1}{\trunc{Ts}^2} \sum_{t=1}^{ \trunc{Ts} } Y_{T,t-1}^2.
\end{eqnarray*}
The term $ \widetilde{V}_{1,T} $ can be treated as in the proof of Theorem~\ref{ThBasic},
namely,
\begin{eqnarray*}
  \widetilde{V}_{1,T}(s) &=&
    \frac{1}{2 \rho_T \trunc{Ts}} \sum_{t=1}^{\trunc{Ts}} ( Y_{T,t}^2 - Y_{T,t-1}^2 )
     K( ( \trunc{Ts}-t )/h ) \\
   &=& \frac{\zeta}{2s} \int_0^s U_T^2(r) K'(\zeta(s-r) ) \,  dr + o_P(1),
\end{eqnarray*}
From the proof of Theorem~\ref{ThBasic} we know that due to (E1)
$$
  \sup_s \left| \widetilde{V}_{3,T}(s) + \frac{ \sigma^2 }{2 s} \int_0^s K(\zeta(s-r)) \, dr
  \right| \stackrel{L_2}{\to} 0,
$$
as $ T \to \infty $. Consider now $ \widetilde{V}_{2,T} $. By definition of $ \rho_T $
we obtain
\begin{eqnarray*}
  \widetilde{V}_{2,T} &=& \frac{1-\rho_T^2}{2 \rho_T} \frac{1}{\trunc{Ts}}
    \sum_{t=1}^{\trunc{Ts}} Y_{T,t-1}^2 K( ( \trunc{Ts} - t )/h ) \\
   &=&
     \frac{-2a -a^2/T}{2(1+a/T)} \frac{1}{T \trunc{Ts}}
      \sum_{t=1}^{\trunc{Ts}} Y_{T,t-1}^2 K( ( \trunc{Ts} - t )/h ), \\
   &=&
     -(a/s) \eta^2 \int_0^s \calZ_a^2(r) K( \zeta(s-r) ) \, d r
     + \frac{\eta^2}{2s} K(0) \calZ_a^2(s) +  o_P(1),
\end{eqnarray*}
where due to (K2) the $ o_P(1) $ term is uniform in $ s \in (0,1] $.
Hence, $ \widetilde{V}_{1,T} $, $ \widetilde{V}_{2,T} $, and
$ \widetilde{W}_T $ are functionals of $ U_T $ up to terms of order $ o_P(1) $.
Consequently, joint weak convergence of
$ ( \widetilde{V}_{1,T}, \widetilde{V}_{2,T}, \widetilde{V}_{3,T}, \widetilde{W}_T  ) $
can be shown along the lines of the proof of Theorem~\ref{ThBasic}, and the CMT
yields the result.
\end{proof}

\section{Simulations}
\label{SecSim}

To investigate the statistical properties of the proposed monitoring procedure
we performed a simulation study.
We used the following ARMA(1,1) simulation model. Suppose
$$
  Y_{t+1} = \rho Y_t + e_t - \beta e_{t-1}, t = 1,2, \dots, T = 250,
$$
where $ Y_0 = 0 $, $ \{ e_t \} $ is a sequence of independent $
N(0,1) $-distributed error terms, and $ \rho $ and $ \beta $ are
parameters. We investigated the cases given by $ \rho = 1, 0.98,
0.95, 0.9 $ and $ \beta = -0.8, 0.5, 0, 0.5, 0.8  $. Clearly, $ \rho
= 1 $ corresponds to the unit root null hypothesis. For $ \beta = 0
$ the innovation terms are uncorrelated corresponding to $ \vartheta
= 1$. This simulation model was also used in Steland (2006), where a
monitoring procedure based on the KPSS unit root test is studied in
detail. Since part of the parameter settings used below are
identical, the results of the present numerical study can be
compared with the corresponding results in Steland (2006).

To study the monitoring rules with estimated control limits
critical values for a significance level
of $ \alpha = 5\% $ were taken from the limit process
defined in (\ref{DefD}) with estimated nuisance parameter.
To down-weight past contributions a Gaussian kernel with bandwidth $ h = 25 $
was used. The nuisance parameter $ \vartheta $ was estimated by the Newey-West estimator
at time point $t$ with lag truncation parameter $m$ chosen by
$ m = m_t = \trunc{ 4 (t/100)^{1/4} } $, $ t = k, \dots, N $. The start of
monitoring, $k$, affects the properties and has to be chosen carefully.
For the rule $ \widehat{S}_T $ we used $ k = 50 $, whereas for
$ \widehat{\widetilde{S}}_T $ a larger value, $ k = 75 $, yielded better results.

To investigate the properties of the monitoring rule, we estimate
empirical rejection rates
of the test which rejects the unit root null hypothesis if the procedure gives a signal,
the average delay, and the  average conditional delay given a signal.
For the detection rule $ \widehat{S}_T $ the ARL is defined by $ E( \widehat{S}_T ) - k + 1$.
We define the CARL as $ E( \widehat{S}_T | k \le \widehat{S}_T \le T ) - k + 1$. The definitions
for $ \widehat{\widetilde{S}}_T $ are analogous. Note that the conditional delay is
very informative under the alternative, since it informs us how quick the method
reacts if it reacts at all. In the tables average delays are given in brackets
and conditional delay in parentheses.

Table~\ref{Tab1} provides the results for the monitoring
procedures $ \widehat{S}_T $ and $ \widehat{\widetilde{S}}_T $
using estimated control limits. The curves $ c( \vartheta ) $ were
obtained by simulating from the limit laws. Overall, $
\widehat{S}_T $ performed well. The performance of the $t$-type
procedure is disappointing. When inspecting the CARL values, the
results seem to be mysterious. E.g. when comparing the CARL for $
\rho = 0.95 $ and $ \rho = 0.9 $ if $ \beta = 0 $, the procedure
seems to misbehave. To explore the reason, Figure~\ref{Fig1}
provides a part of the distribution of $ \widehat{S}_T - k + 1 $.
It can be seen that the percentage of simulated trajectories
leading to immediate detection increases considerably, but the
contribution of these cases to the calculation of the CARL is
negligible. The other trajectories yielding a signal are hard to
detect, and the signals are spread over the remaining time points
with many late signals, which suffice to yield large CARL values.
This fact shows that a single number as the CARL can not
summarized the statistical behavior sufficiently. It highlights
the benefit that the random walk null hypothesis can often be
rejected very early.

The simulation results for the control charts using transformed
statistics are summarized in Table~\ref{Tab2}. Here we used exact
control limits obtained by simulation using 20,000 repetitions.
Comparing the transformation control statistics with these control
limits yields quite accurate results if $ \beta = 0$. The $t$-type
version is preferable for $ \beta < 0 $.

Comparing the methods $ \widehat{S}_T $ (using estimated control limits)
and $ Z_T $ (using transformed statistics), our results indicate that the more
computer-intensive approach to use estimated control limits provides more accurate
results.

\begin{table}
\begin{tabular}{cccccc} \hline
 $\rho$ & \multicolumn{5}{c}{$\beta$} \\
        &   $ -0.8$ & $-0.5$ & $0.0$   & $0.5$   & $0.8$ \\ \hline
\\
\multicolumn{6}{l}{{\em Weighted DF control chart with estimated control limits,
$ \widehat{S}_T $}} \\
$1$    & $0.024$ & $0.025$ & $0.036$ & $0.154$ & $0.56$\\
%       & $(9.1)$ & $(10)$ & $(6.9)$ & $(7.6)$ & $(8.8)$\\
%       & $[195.7]$ & $[195.6]$ & $[193.4]$ & $[170.5]$ & $[92.9]$\\
$0.98$ & $0.043$ & $0.044$ & $0.062$ & $0.264$ & $0.835$\\
       & $(11.9)$ & $(12)$ & $(9)$ & $(11.3)$ & $(13)$\\
       & $[192.1]$ & $[192]$ & $[188.3]$ & $[150.3]$ & $[43.9]$\\
$0.95$ & $0.095$ & $0.098$ & $0.129$ & $0.5$ & $0.991$\\
       & $(22.3)$ & $(20.8)$ & $(15.3)$ & $(17.8)$ & $(6.7)$\\
       & $[183.4]$ & $[182.5]$ & $[176.3]$ & $[109]$ & $[8.5]$\\
$0.9$  & $0.3$ & $0.306$ & $0.36$ & $0.877$ & $1$\\
       & $(39.2)$ & $(36.9)$ & $(28.4)$ & $(18.8)$ & $(1.5)$\\
       & $[151.9]$ & $[150.3]$ & $[138.2]$ & $[41.1]$ & $[1.5]$\\
\\
\multicolumn{6}{l}{{\em $t$-type version $ \widehat{\widetilde{S}}_T $.}} \\
$1$     & $0.017$ & $0.018$ & $0.047$ & $0.301$ & $0.763$\\
% $(85.3)$ & $(71.8)$ & $(11.3)$ & $(3.6)$ & $(2.6)$\\
% $[173.8]$ & $[173.5]$ & $[167.6]$ & $[123.6]$ & $[43.4]$\\
$0.98$  & $0.007$ & $0.01$ & $0.092$ & $0.538$ & $0.972$\\
        & $(20.1)$ & $(16.6)$ & $(5.1)$ & $(4.2)$ & $(2.3)$\\
        & $[174.1]$ & $[173.7]$ & $[159.6]$ & $[83.1]$ & $[7.1]$\\
$0.95$  & $0.014$ & $0.024$ & $0.217$ & $0.835$ & $1$\\
        & $(6.9)$ & $(6)$ & $(6)$ & $(4.5)$ & $(1.1)$\\
        & $[172.8]$ & $[171.1]$ & $[138.4]$ & $[32.6]$ & $[1.1]$\\
$0.9$   & $0.064$ & $0.106$ & $0.545$ & $0.99$ & $1$\\
        & $(8)$ & $(7)$ & $(6.9)$ & $(2.1)$ & $(1)$\\
        & $[164.4]$ & $[157.3]$ & $[83.5]$ & $[3.8]$ & $[1]$\\
\hline
\end{tabular}
\caption{Results for the weighted DF control chart with estimated control limits, $ \widehat{S}_T $.}
\label{Tab1}
\end{table}

\begin{figure}
\includegraphics[width=8cm]{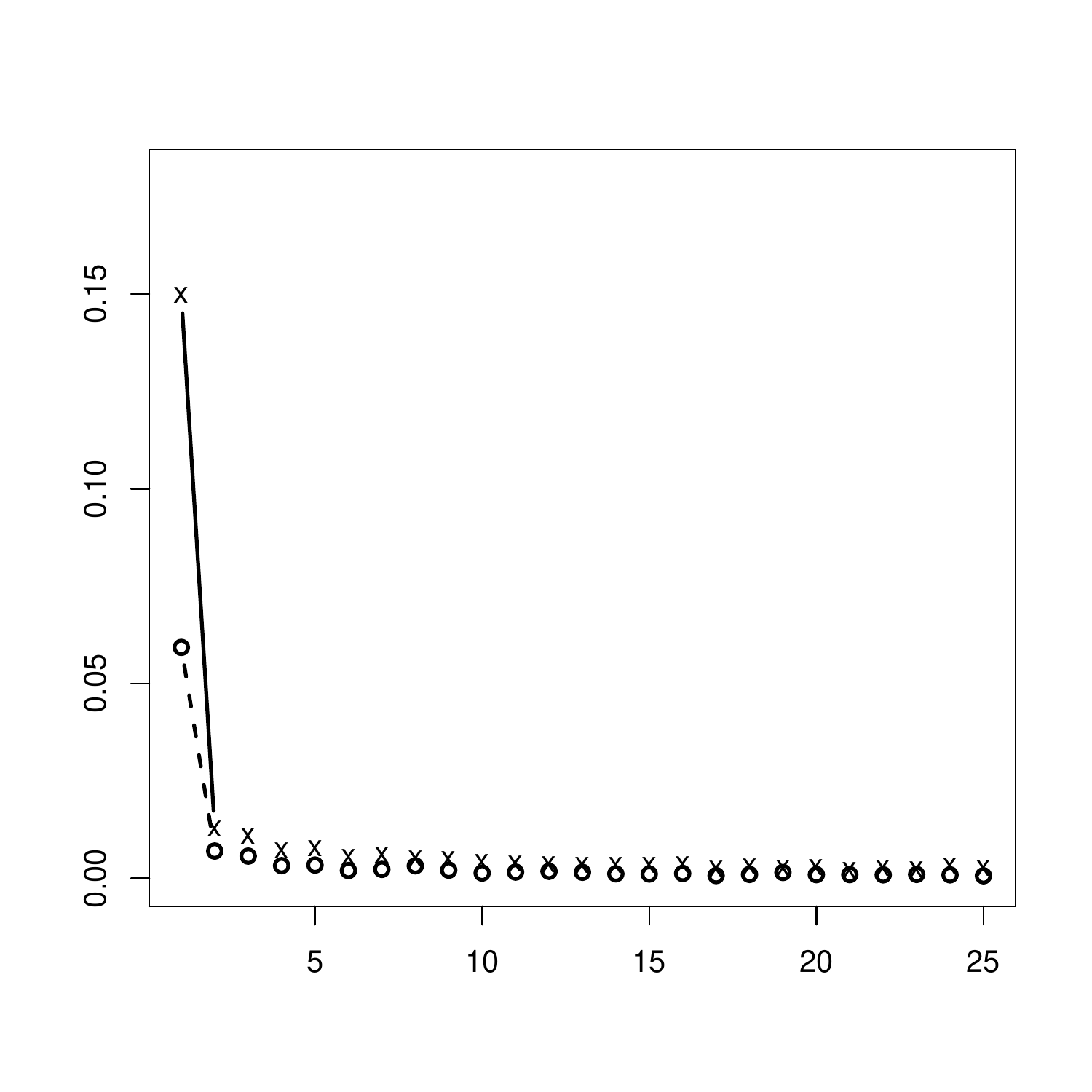}
\caption{Part of the distribution of $ \widehat{S}_T - k + 1 $ for
$ \rho = 0.95 $ (circles) and $ \rho = 0.9 $ (crosses).}
\label{Fig1}
\end{figure}

\begin{table}
\begin{tabular}{cccccc} \hline
 $\rho$ & \multicolumn{5}{c}{$\beta$} \\
        &   $ -0.8$ & $-0.5$ & $0.0$   & $0.5$   & $0.8$ \\ \hline
\\
\multicolumn{6}{l}{{\em transformed weighted DF control chart $ Z_T $}} \\
$1$     & $0.02$ & $0.02$ & $0.032$ & $0.193$ & $0.677$\\
%        & $(5.4)$ & $(5.9)$ & $(7.3)$ & $(17.3)$ & $(21.3)$\\
%        & $[197.1]$ & $[196.9]$ & $[194.7]$ & $[165.4]$ & $[79.3]$\\
$0.98$  & $0.031$ & $0.032$ & $0.055$ & $0.352$ & $0.949$\\
        & $(8.9)$ & $(7.3)$ & $(10.4)$ & $(29.4)$ & $(21.8)$\\
        & $[194.9]$ & $[194.6]$ & $[190.3]$ & $[140.4]$ & $[30.8]$\\
$0.95$  & $0.069$ & $0.066$ & $0.118$ & $0.684$ & $0.998$\\
        & $(12.5)$ & $(12.5)$ & $(17.2)$ & $(39.2)$ & $(10)$\\
        & $[187.9]$ & $[188.4]$ & $[179.1]$ & $[90.2]$ & $[10.4]$\\
$0.9$   & $0.199$ & $0.211$ & $0.355$ & $0.965$ & $1$\\
        & $(24.2)$ & $(23.9)$ & $(32.1)$ & $(24.2)$ & $(4.8)$\\
        & $[165.6]$ & $[163.3]$ & $[140.8]$ & $[30.5]$ & $[4.9]$\\
\\
\multicolumn{6}{l}{{\em $t$-type version $ \widetilde{Z}_T $}} \\
$1$     & $0.059$ & $0.035$ & $0.041$ & $0.439$ & $0.952$\\
% $(9.3)$ & $(3.7)$ & $(4.2)$ & $(9.7)$ & $(5.3)$\\
% $[189.6]$ & $[194]$ & $[192.9]$ & $[116.9]$ & $[14.7]$\\
$0.98$  & $0.1$ & $0.06$ & $0.073$ & $0.683$ & $0.999$\\
        & $(4.3)$ & $(3.8)$ & $(4.5)$ & $(12.5)$ & $(2.5)$\\
        & $[181.2]$ & $[189.1]$ & $[186.5]$ & $[72.2]$ & $[2.6]$\\
$0.95$  & $0.194$ & $0.113$ & $0.152$ & $0.914$ & $1$\\
        & $(4.9)$ & $(4.6)$ & $(5.6)$ & $(11.6)$ & $(1.2)$\\
        & $[163]$ & $[178.8]$ & $[171.2]$ & $[28]$ & $[1.2]$\\
$0.9$   & $0.427$ & $0.294$ & $0.365$ & $0.996$ & $1$\\
        & $(5.5)$ & $(5.2)$ & $(6.9)$ & $(4.7)$ & $(1)$\\
        & $[117.5]$ & $[143.3]$ & $[130.2]$ & $[5.4]$ & $[1]$\\
\hline
 \end{tabular}
\caption{Results for the transformed weighted DF control charts $ Z_T $ and
$ \widetilde{Z}_T $.}
\label{Tab2}
\end{table}

\section*{Acknowledgements}

The support of Deutsche Forschungsgemeinschaft (SFB 475, {\em
Reduction of Complexity in Multivariate Data Structures}) is
gratefully acknowledged. I thank Dipl.-Math. Sabine Teller for
proof-reading.

\end{document}